\documentclass{article}
\usepackage{fancyhdr}
\usepackage{enumerate}
\usepackage{etaremune}
\usepackage{graphicx}
\usepackage{titlesec}
\usepackage{lipsum}
\usepackage{multicol}
\usepackage{amsmath}
\usepackage{amssymb}
\usepackage{mathtools} 
\usepackage{amsthm}
\usepackage[utf8]{inputenc}
\usepackage{xcolor}
\usepackage{booktabs, comment} 
\usepackage{pgfpages}
\usepackage{csquotes}
\usepackage{tikz}
\usetikzlibrary{fadings}
\usepackage{color,soul}
\usepackage[normalem]{ulem}
\usepackage[margin=1.5in]{geometry}
\usepackage{bm}
\usepackage{hyperref}

\newcommand{\Z}{\mathbb{Z}}

\newcommand{\R}{\mathbb{R}}
\newcommand{\U}{\mathcal{U}}
\newcommand{\E}{\mathcal{E}}

\newcommand{\cH}{\mathcal{H}}

\newtheorem{thm}{Theorem}
\newtheorem{prop}{Proposition}
\newtheorem{observation}{Observation}
\newtheorem{lemma}[thm]{Lemma}
\newtheorem{claim}[thm]{Claim}

\theoremstyle{definition}

\title{Pancyclicity in hypergraphs with large uniformity}
\author{
{{Teegan Bailey}}\thanks{
\footnotesize {Email: {\tt teeganb@email.sc.edu}.
}}
\and{{Isaiah Hollars}}\thanks{
\footnotesize {Email: {\tt isaiah.hollars@sc.edu}.
}}
\and{{Yupei Li}}\thanks{
\footnotesize {E-mail: {\tt yupei@email.sc.edu}.
}}
\and{{Ruth Luo}}\thanks{
\footnotesize {Email: {\tt ruthluo@sc.edu}.
}}}
\date{\today}

\begin{document}
\vspace{-0.3in}

\maketitle

\begin{abstract}
A Berge cycle of length $\ell$ in a hypergraph $\mathcal{H}$ is a sequence of alternating vertices and edges $v_0e_0v_1e_1...v_\ell e_\ell v_0$ such that $\{v_i,v_{i+1}\}\subseteq e_i$ for all $i$, with indices taken modulo $\ell$. 


For $n$ sufficiently large and $r\geq \lfloor\frac{n-1}{2}\rfloor-1$ we prove exact  minimum degree conditions for an $n$-vertex, $r$-uniform hypergraph to contain Berge cycles of every length between $2$ and $n$. 
In conjunction with previous work, this provides sharp Dirac-type conditions for pancyclicity in  $r$-uniform hypergraphs for all $3\leq r\leq n$ when $n$ is sufficiently large.

\end{abstract}

\section{Introduction}
    A {\em hypergraph} $\cH$ is a set of vertices $V(\cH)$ and a set of edges $E(\cH)$ such that every edge is a subset of the vertices. We say a hypergraph is {\em $r$-uniform} if every edge contains exactly $r$ vertices. In particular, a $2$-uniform hypergraph is simply a graph.
    The {\em degree} of a vertex $v$ in a hypergraph $\cH$, denoted $d_{\cH}(v)$, is the number of edges containing $v$. We denote by $\delta(\cH)$ the {\em minimum degree} over all vertices in $\cH$. Given a pair of vertices $u$ and $v$, the \textit{co-degree} of the pair, denoted $d_{\cH}(u,v)$, is the number of edges containing both $u$ and $v$. When $\cH$ is clear from context, we may drop the subscripts in the definitions above. 
    
    A {\em hamiltonian cycle (path)} in a graph is a cycle (path) which covers all of its vertices. We say a graph is {\em hamiltonian} if it contains such a cycle. The first extremal result for hamiltonian cycles was due to Dirac.
    
    \begin{thm}[Dirac~\cite{Dirac}]
    Let $n \geq 3$. If $G$ is an $n$-vertex graph with minimum degree $\delta(G) \geq n/2$, then $G$ is hamiltonian.
    \end{thm}
    
    Bondy~\cite{Bondy} studied a much stronger property in graphs called {\em pancyclicity}. An $n$-vertex graph is {\em pancyclic} if it contains cycles of every length from $3$ to $n$. The following is one of many results by Bondy on pancyclicity. In particular, it strengthens Dirac's Theorem.
    
    \begin{thm}[Bondy~\cite{Bondy}]\label{Bondy}
    Let $G$ be an $n$-vertex hamiltonian graph with $|E(G)| \geq \frac {n^2}4$. Then either $G$ is pancyclic or $K_{\frac n2, \frac n2}$.
    \end{thm}
    
    Thus the same minimum degree condition in Dirac's Theorem not only yields a cycle of length $n$ but rather cycles of {\em all} lengths, unless $G$ is the complete bipartite graph which misses all odd cycles.
    
    In~\cite{Bondy}, Bondy proposed a so-called meta-conjecture: {\em ``almost any
    nontrivial condition on a graph which implies that the graph is hamiltonian also implies that the
    graph is pancyclic.” } He later added {\em ``...there may be a simple family of exceptional graphs."}
    
    See for example~\cite{Cream, Draganic, Letzter, Schmeichel} for some results in support of Bondy's meta-conjecture. 
    
    The goal of this paper is to extend pancyclic results to hypergraphs. First we define our notion of cycles in hypergraphs---while many definitions of cycles in hypergraphs exist, e.g., tight cycles, loose cycles, we consider a general definition called {\em Berge cycles}.
    
    A {\em Berge cycle} of length $k$ in a hypergraph is a set of $k$ distinct vertices $\{v_0, v_1, \ldots, v_{k-1}\}$ and $k$ distinct edges $\{e_0, \ldots, e_{k-1}\}$ such that for all $0\leq i \leq k-1$, $\{v_i, v_{i+1}\} \subseteq e_i$. For convenience, we write Berge cycles as an alternating sequence of vertices and edges $v_0e_0v_1\ldots v_{n-1}e_{n-1}v_0$. We often say {\em $k$-Berge cycle} to refer to a Berge cycle of length $k$.

    A {\em hamiltonian Berge cycle} in an $n$-vertex hypergraph is a Berge cycle of length $n$. We call a hypergraph $\cH$ {\em hamiltonian} if it contains a hamiltonian Berge cycle. Moreover, we say $\cH$ is {\em pancyclic} if it contains Berge cycles of every length from $2$ to $n$. Note that unlike in the graph case, $r$-uniform hypergraphs may have $2$-Berge cycles. 
    
    Sufficient conditions for long Berge cycles and Berge cycles of fixed length have been well studied in a number of papers. See for instance~\cite{GL, FKL1, FKL2, KL, Gyori, salia, MHG, CP}. In particular, the following theorem gives a Dirac-type condition for the existence of hamiltonian Berge cycles in uniform hypergraphs. Note that the minimum degree required depends on the uniformity $r$, and there is a transition around $r = n/2$.  
        \begin{thm}[Kostochka, Luo, McCourt~\cite{KLM}]\label{KLMthm}
            Let $3 \leq  r <n$, and let $\cH$ be an $n$-vertex $r$-uniform hypergraph. If $r\leq \left\lfloor (n-1)/2 \right\rfloor$ and $\delta(\cH) \geq  \binom{\left\lfloor (n-1)/2 \right\rfloor}{r-1} +1$ or  $r\geq n/2$ and $\delta(\cH) \geq r$, then $\cH$ contains a Berge hamiltonian cycle. 
        \end{thm}

        In~\cite{BLL}, three of the present authors proved that for $r \leq \lfloor (n-1)/2 \rfloor -2$, the minimum degree condition in Theorem~\ref{KLMthm} in fact guarantees the stronger condition of pancyclicity when $n$ is sufficiently large. 
    
        \begin{thm}[Bailey, Li, Luo~\cite{BLL}]
            Let $n$ and $r$ be positive integers with $3 \leq  r \leq  \lfloor (n-1)/2 \rfloor - 2$, such that $n\geq 70$ if $r \in \{3,4\}$ and $n \geq 39$ if $r \geq 5$. If $\cH$ is an $n$-vertex, $r$-uniform hypergraph with minimum degree $\delta(\cH) \geq  \binom{\lfloor (n-1)/2 \rfloor}{r-1} + 1$ then $\cH$ is pancyclic.       
        \end{thm}
    
    The goal of this paper is to complete the remaining cases, i.e., $r \geq \lfloor (n-1)/2 \rfloor -1$.  The minimum degree bounds are the same as in Theorem~\ref{KLMthm}, thus we can view our theorem as a strenghtening of the hamiltonian result, akin to Bondy's meta-conjecture.    
    \begin{thm}\label{main}
        Let $n$ and $r$ be integers with $r \geq \lfloor (n-1)/2 \rfloor - 1$. Let $\cH$ be an $n$-vertex $r$-uniform hypergraph. If
        \begin{enumerate}[(a)]
            \item $n \geq 31$, $r\leq \lfloor (n-1)/2 \rfloor$ and 
            $\delta(\cH) \geq  \binom{\left\lfloor (n-1)/2 \right\rfloor}{r-1} +1$ or 
            \item $n \geq 55, r\geq n/2$ and $ \delta(\cH) \geq r$,
        \end{enumerate}
        then $\cH$ is pancyclic.
    \end{thm}
    The bounds for $\delta$ and $r$ are sharp due to the following constructions. We note that we do not think the bounds $n \geq 31$ and $n \geq 55$ are best possible, and they can likely be improved with more careful arguments.
    
    {\bf Construction 1}. Let $r \leq \lfloor (n-1)/2\rfloor$. Let $\cH_1$ consist of two cliques on vertex sets $V_1$ and $V_2$ where $|V_1|=|V_2| = (n+1)/2$ and $|V_1 \cap V_2| = 1$ if $n$ is odd, and $|V_1| = |V_2| = n/2$ and $|V_1 \cap V_2| = 0$ if $n$ is even. In addition, when $n$ is even, $\cH_1$ may have one extra edge joining $V_1$ and $V_2$.
    
    {\bf Construction 2}.  Let $r \leq \lfloor (n-1)/2\rfloor$. Let $V(\cH_2)$ be partitioned into sets $V_1$ and $V_2$ with $|V_1| = \lfloor (n-1)/2 \rfloor, |V_2| = \lceil (n+1)/2 \rceil$. The edge set $E(\cH_2)$ contains all edges with at most one vertex in $V_2$. Moreover, when $n$ is even, we may add one additional edge containing multiple vertices in $V_2$. 
    
    {\bf Construction 3}. Let $r \geq n/2$. Let $\cH_3'$ be any $r$-uniform, $r$-regular, $n$-vertex hypergraph (such hypergraphs exist, for instance the tight $n$-cycle). Let $\cH_3$ be obtained from $\cH_3'$ by removing exactly one edge.
    
    All three constructions do not have hamiltonian Berge cycles---$\cH_1$ has either a vertex or an edge whose removal disconnects $\cH_1$; $\cH_2$ does not have enough edges connecting multiple vertices in $V_2$ to yield a spanning Berge cycle; $\cH_3$ contains only $n-1$ edges. Since these examples are not hamiltonian, they also fail to be pancyclic. Moreover, one can check that the minimum degrees of the graphs are one less than those in Theorem~\ref{main}.
    
\section{Proof outline and tools}

    Let $\cH $ be an $r$-uniform hypergraph with $V(\mathcal{H}) = \{v_0,\dots, v_{n-1}\}$ and $r \geq \lfloor (n-1)/2 \rfloor -1$. Throughout, we will have the necessary degree conditions to apply Theorem \ref{KLMthm}, which guarantees the existence of a Berge hamiltonian cycle. Let $C = v_0e_0v_1 \ldots v_{n-1}e_{n-1}v_0$ be such a Berge cycle where $\{v_i, v_{i+1}\} \subset e_i$ for all $i$.

             {\bf Remark}: Throughout this paper, all arithmetic for vertex indices are taken modulo $n$. We will compare vertex indices in $\Z/n\Z$ with $<$. For $\bar{a},\bar{b}\in \Z$, there exists a unique choice of $a,b\in \{0,\dots,n-1\}$ such that $\bar{a} = a+n\Z$ and $\bar{b} = b+n\Z$. We say $\bar{a}< \bar{b}$ if $a<b$. For readability, we omit the use of bar notation when referring to the indices. 

In any $r$-uniform hamiltonian hypergraph with $r \geq 3$, there is a vertex $v$ with degree at least $3$. Since we have $r > (n-1)/3$, there exists a pair of edges containing $v$ that intersect in another point $u$. This yields a Berge cycle of length $2$. In the rest of the paper, we focus our attention on finding Berge cycles with lengths in $\{3, \ldots, n\}$.
    
    Let $\mathcal{E}= E(\mathcal{H})- E(C)$ denote the set of edges in $\cH$ that are not used in the Berge hamiltonian cycle $C$. We will often refer to these as the {\em extra edges} of $\cH$. For $0\leq i\leq n-1$, let $\mathcal{E}_i = \{e\in \mathcal{E}: v_i\in e\}$. Define $\mathcal{U} = \bigcup_{e\in \mathcal{E}} e$ to be the set of vertices contained in the edges of $\E$ and  
    similarly define $\mathcal{U}_i= \bigcup_{e\in\mathcal{E}_i} e$. 
    
    In some parts of the proofs, we consider an auxiliary graph where cycles in the auxiliary graph correspond to Berge cycles in the hypergraph. We can then make use of the following pancyclic-type results for graphs.
    
    A graph $G$ is called {\em weakly pancyclic} if it contains every cycle length from its girth to its circumference. 
    \footnote{The {\em girth} of a graph $G$, $g(G)$, is the length of its shortest cycle. The {\em circumference} of $G$, $c(G)$, is the length of its longest cycle. In the case where $g(G) = 3$ and $c(G) = n$, being weakly pancyclic is equivalent to being pancyclic.} 
    
    \begin{thm}[Brandt~\cite{Brandt}]\label{Brandt}
        Every nonbipartite graph $G$ of order $n$ with minimum degree at least $\frac{n+2}{3}$ is weakly pancylic, with girth either 3 or 4.
    \end{thm}
    
    We say a bipartite graph $G$ is {\em bipancyclic} if it contains cycles of every even length from $4$ to $|V(G)|$. Similarly, we say $G$ is \textit{weakly bipancyclic} if it contains cycles of every even length from its girth to its circumference.
    \begin{thm}[Hu, Sun \cite{HuSun}]\label{HuSun1}
        Every hamiltonian bipartite graph on $2n$ vertices with minimum degree at least $\frac{n}{3}+4$ is bipancyclic.
    \end{thm}
    In~\cite{HuSun}, a stronger version is proved which will be useful for us.
    \begin{thm}[Hu, Sun \cite{HuSun}]\label{HuSun2}
        If $G=(X\sqcup Y, E)$ is a bipartite graph with minimum degree at least $\frac{m}{3}+4$, where $m = \max\{|X|,|Y|\}$, then $G$ is a weakly bipancyclic graph with girth 4.
    \end{thm}
    
    See~\cite{BFG,BT} for more results on weakly pancyclic graphs and~\cite{SM, Amar,Zhang} for more results on bipancyclic graphs.

    We will prove a stronger version of Theorem~\ref{main} for the cases $r \geq \lceil (n-1)/2 \rceil$. The stronger statements of the following two theorems along with Theorem~\ref{KLMthm} will imply Theorem~\ref{main}. 

    \begin{thm}\label{mainbig} 
        Let $n$ and $r$ be integers with $r \geq \lfloor(n-1)/2 \rfloor$. Let $\cH$ be an $r$-uniform hypergraph with a hamiltonian Berge cycle $C$. 
        Define \[c_r =   \left\{
        \begin{array}{ll}
             6 & \text{if } r = \lfloor (n-1)/2 \rfloor, \\
             1 & \text{if } r \geq n/2. 
        \end{array} 
        \right. \]
        
        If there exists a vertex $v$ contained in at least $c_r$ edges outside of $E(C)$, then $\cH$ is pancyclic.
    
   \end{thm}

   \begin{thm}\label{mainsmall}
   
        Let $n \geq 19$ be an integer and set $r = \lfloor(n-1)/2 \rfloor - 1$. Let $\cH$ be an $r$-uniform hypergraph with a hamiltonian Berge cycle $C$. If every vertex in $\cH$ is contained in at least $5(r-1)+2$ edges outside of $E(C)$, then $\cH$ is pancyclic. 
       
   \end{thm}
    
    \begin{proof}[Proof of Theorem~\ref{main} assuming Theorems~\ref{mainbig} and~\ref{mainsmall}.] 
        By Theorem~\ref{KLMthm}, $\cH$ has a hamiltonian Berge cycle $C$.
        If $r = \lfloor (n-1)/2 \rfloor -1$ and $\delta(\cH) \geq \binom{\lfloor (n-1)/2 \rfloor}{r-1}+1$, then each vertex is contained in at least $\binom{\lfloor (n-1)/2 \rfloor}{r-1}+1 - n = \frac{1}{2}\cdot\lfloor (n-1)/2 \rfloor \lfloor (n-3)/2 \rfloor + 1 - n$ edges outside of $E(C)$. This is at least $5(r-1)+2$ when $n \geq 31$, so we may apply Theorem~\ref{mainsmall}.
    
        Now suppose $r \geq \lfloor (n-1)/2 \rfloor$ and
        fix an $r$-uniform $n$-vertex hypergraph $\cH$ satisfying the statement of Theorem~\ref{main}. For $r = \lfloor (n-1)/2 \rfloor$, $\delta(\cH) \geq \binom{\lfloor (n-1)/2 \rfloor}{r-1} +1 = r+1$, and for $r \geq n/2$, $\delta(\cH) \geq r$. In either case, $\delta(\cH) \geq n/2$.

        Set $X = V(\cH)$ and $Y = E(\cH)$. The \textit{incidence bipartite graph} of $\cH$ is the bipartite graph $G=X\sqcup Y$ where $v_ie\in E(G)$ if and only if $v_i\in e$. A cycle of length $2\ell$ in $G$ yields a Berge cycle of length $\ell$ in $\cH$, and vice versa. Thus if $G$ is weakly bipancyclic with girth $4$ and circumference $2n$, then $\cH$ is pancyclic. By Theorem~\ref{KLMthm}, $G$ has a cycle of length $2n$. 
        
        Observe that if $|E(\cH)| = m$, then $G$ has $m+n$ vertices, where $m \geq n$. Moreover, $d_G(x) \geq \delta(\cH) \geq n/2$ and  $d_G(y) = r$ for all $x \in X$ and  $y \in Y$. Thus $\delta(G) \geq \lfloor (n-1)/2\rfloor$. 
        
        By Theorem~\ref{HuSun2}, if $\delta(G) \geq m/3 + 4$, then $G$ is weakly pancyclic with girth $4$. That is, $\cH$ is pancyclic. So suppose $\lfloor (n-1)/2\rfloor \leq \delta(G) \leq m/3 + 4$. This implies $m \geq 3n/2 - 15$. Hence the set of extra edges, $\E = E(\cH) - E(C)$, contains at least $n/2 - 15$ edges. Since each edge has $r$ vertices, there exists a vertex with $d_{\E}(v) \geq \frac{r(n/2-15)}{n} \geq  c_r$ for $n \geq 55$. Thus we may apply Theorem~\ref{mainbig} to obtain that $\cH$ is pancyclic.
    \end{proof}

    In the remainder of the paper we prove Theorems~\ref{mainbig} and~\ref{mainsmall}. We handle the cases of $r = \lfloor (n-1)/2 \rfloor-1$, $r = \lfloor (n-1)/2 \rfloor$, and $r \geq n/2$ separately. Equivalently, these are the cases $n \in \{2r +3, 2r+4\}$, $n \in \{2r+1, 2r+2\}$, and $n \leq 2r$. A common tool will be the shifting functions described in the next subsection.

    \subsection{Shifting functions}
         Recall that $\cH$ has hamiltonian cycle $C= v_0e_0v_1...v_{n-1}e_{n-1}v_1$.
         Let $A\subset V(\mathcal{H})$. For $s\in \{0,\dots,n-1\}$, define the shifted set 
         \[
            A^{+s} = \{v_{i+s}: v_i\in A\},
        \]
    where addition of indices is performed modulo $n$. For $e\in \E$, we say that $e$ contains a \textit{$k$-chord} if $\{v_i, v_{i+k}\}\subset e$ for some $v_i$.
        \begin{observation}\label{chordsgivecycles}
             If an extra edge $e\in \mathcal{E}$ contains a $k$-chord $\{v_i,v_{i+k}\}$, then $\mathcal{H}$ contains a $(k+1)$-Berge cycle. 
        \end{observation}

        Indeed, the Berge cycle $v_i e_{i} \dots e_{i+k-1} v_{i+k} ev_i$ has length $k+1$. 

    
         For $s\in \{0,\dots, n-1\}$ define the shifting function $S_s: V(\cH)\rightarrow V(\cH)$ by
        \[
            S_s(v_i) = \begin{cases}
            		      v_{i+s}& i+s\leq n-1\\
            		      v_{i+s+1}& i+s\geq n
            	       \end{cases}.
        \]
         We call $s$ the \textit{shifting constant}. For a set $A \subseteq V(\cH)$, we denote $S_s(A) = \{S_s(v_i):v_i \in A\}$. 
         
        \begin{lemma}[Properties of the shifting function $S_s$]\label{Sshift} 
            Fix $s\in \{0,1,...,n-1\}$. Then 
            \begin{enumerate}[(a)]
                \item \label{shiftInjective}$|S_s(A)| \geq |A|-1$ for any $A\subseteq V(\mathcal{H})$ with equality if and only if $v_0,v_{n-1}\in A$ and 
                \item \label{shiftCycle} if $S_s(\mathcal{U}_0)\cap e_0 \neq \emptyset$, then $\mathcal{H}$ has an $(n-s+1)$-Berge cycle (provided $1\leq s\leq n-2$). 
            \end{enumerate} 
        \end{lemma}
        \begin{proof}
            Fix $s\in \{0,\dots,n-1\}$. To establish part $(a)$, note that $S_s(v_0) = v_s = S_s(v_{n-1})$. On the other hand, $S_s$ is injective when restricted to $\{v_0, \dots,v_{n-2}\}$. Thus $(a)$ holds.
            
            We now prove part $(b)$. Suppose $v_i \in S_s(\U_0) \cap e_0$. Then there exists $v_j\in \U _0$ such that $S_s(v_j)=v_i\in e_0$. By the definition of $\U_0$ there exists an extra edge $f\in \E_0$ with $v_0,v_j\in f$. We consider two cases.
            
            \textit{Case 1: $v_i=v_{j+s}$.} By the definition of the shifting function $S_s$, $j+s\leq n-1$. When $j\neq 0$, the Berge cycle 
            \[
                v_0fv_{j}e_{j-1}v_{j-1}\dots v_1e_0 v_{j+s} e_{j+s} v_{j+s+1}\dots v_{n-1}e_{n-1}v_0,
            \]
            has length $n-s+1$ since it omits exactly the vertices $v_k$ with $ j< k<j+s$. See Figure~\ref{fig:S_scycle} for reference. If $j=0$, then this collapses to 
            \[
                v_0e_0v_se_{s} v_{s+1}\dots v_{n-1}e_{n-1}v_0,
            \]
            which again is an $(n-s+1)$-Berge cycle. 
            
            \textit{Case 2: $v_i=v_{j+s+1}$.} Note that $j+s\geq n$. When $j\neq n-s$, we obtain the Berge cycle
            \[
                v_0e_0v_{j+s+1}e_{j+s+1}v_{j+s+2}\dots v_{j-1}e_{j-1} v_j fv_0,
            \]
            which has length $n-s+1$ because it omits the vertices strictly between $v_j$ and $v_{j+s+1}$ except for $v_0$. See Figure \ref{fig:S_scycle} for reference. If $j=n-s$, then $v_i=v_1$ and this collapses to 
            \[
                v_0e_0v_1e_{1}v_2\dots v_{n-s-1}e_{n-s-1}v_{n-s} fv_0,
            \]
            which also has length $n-s+1$. This completes the proof. 
        \end{proof}
        \begin{figure}[h]
            \centering
            \begin{tikzpicture}[shorten >=1pt, minimum size=.01cm, auto, node distance=5cm, thick, line cap=rect,
                    main node/.style={circle,inner sep = 2pt, outer sep = 0pt, minimum width = 2pt, fill=black, draw, font=\sffamily\bfseries},
                    main2 node/.style={inner sep = 0.5pt, outer sep = 0pt, fill=white, font=\sffamily}
                    ]
                \def\mylist{0/0, -1/1, -7/j, -9/j+s}
                \def \rone{1.98}
                \def \rtwo{2.35}
                \def \rthree{2}
                \def \rfour{1.85}
                \def \rfive{2.0}
                
                \node[main2 node, below = 3pt] (e0) at ({\rtwo*cos(79)}, {\rtwo*sin(79)}) {$e_0$};
                \draw[thin] (0,0) circle (\rone);
                \foreach \t/\name in \mylist {
                    \node[main2 node] (v\name') at ({\rtwo*cos(24*\t+90)},{\rtwo*sin(24*\t+90)}) {$v_{\name}$};
                    \node[main node] (v\name) at ({\rone*cos(24*\t+90)},{\rone*sin(24*\t+90)}) {};
                }
                \draw[very thick,domain=282:426] plot ({\rone*cos(\x)}, {\rone*sin(\x)});
                \draw[very thick,domain=90:234] plot ({\rone*cos(\x)}, {\rone*sin(\x)});
                \draw[fill = lightgray, line width =0pt] (v0) to[bend left = 10] (vj+s) to[bend right = 10] ({\rfive*cos(-24 +90)}, {\rfive*sin(-24 +90)}) arc (66:90:\rfive);
                \draw[very thick] (v0) to[bend right = 30] node[pos = 0.8, right] {$f$} (vj);
                \foreach \t/\name in \mylist {
                    \node[main node] (v\name) at ({\rone*cos(24*\t+90)},{\rone*sin(24*\t+90)}) {};
                }
                \draw[very thick] (vj+s) to[bend right = 10] (v1);

            \end{tikzpicture}
            \hspace{1em}
            \begin{tikzpicture}[shorten >=1pt, minimum size=.01cm, auto, node distance=5cm, thick, line cap=rect,
                    main node/.style={circle,inner sep = 2pt,  minimum width = 2pt, fill=black,draw,font=\sffamily\bfseries},
                    main2 node/.style={ fill=white,font=\sffamily}
                    ]
                \def\mylist{0/0, -3/j+s+1, 3/j}
                \def \rone{1.98}
                \def \rtwo{2.5}
                \def \rthree{1.98}
                \node[main2 node, below = 3pt] (e0) at ({\rtwo*cos(79)}, {\rtwo*sin(79)}) {$e_0$};
                \foreach \t/\name in \mylist {
                    \node[main2 node] (v\name') at ({\rtwo*cos(24*\t+90)},{\rtwo*sin(24*\t+90)}) {$v_{\name}$};
                    \node[main node] (v\name) at ({\rone*cos(24*\t+90)},{\rone*sin(24*\t+90)}) {};
                }
                \draw[thin] (0,0) circle (\rone);
                \draw[very thick] (v0) to[bend right = 30, below] node[pos = 0.55, below] {} (vj+s+1);
                \draw[very thick, domain=162:378] plot ({\rone*cos(\x)}, {\rone*sin(\x)});
                \draw[very thick] (vj) to[bend right = 30]  node[pos = 0.45, below]{$f$} (v0) ;
                \draw[fill = lightgray, line width = 0pt] (v0) to[bend right = 30] (vj+s+1) to[bend left = 30] ({\rthree*cos(-24 +90)}, {\rthree*sin(-24 +90)}) arc (66:90:\rthree);
                \foreach \t/\name in \mylist {
                    \node[main node] (v\name) at ({\rone*cos(24*\t+90)},{\rone*sin(24*\t+90)}) {};
                }
            \end{tikzpicture}
            \caption{An $(n-s+1)$-Berge cycle in cases 1 and 2 respectively.}
            \label{fig:S_scycle}
        \end{figure}
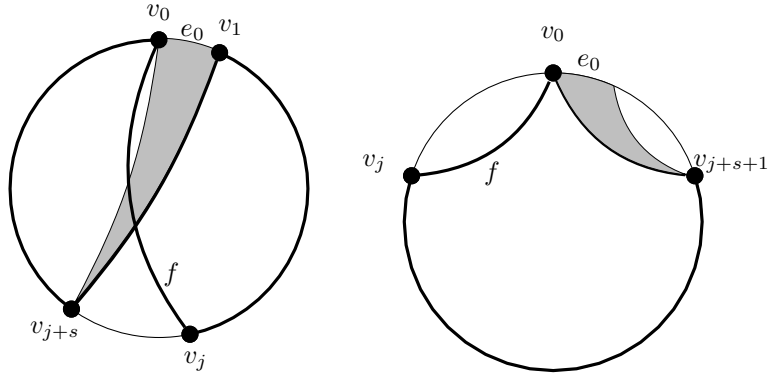
    \subsection{Self-Shift Complementary Sets}
        When forbidding $k$-chords for some $k \in [n-1]$, we will often consider sets $ \mathcal A\subset V(\cH)$ where $v_i\in  \mathcal A$ implies $v_{i+k}\notin \mathcal A$. If $A = \{i: v_i \in \mathcal A\}$ is the set of indices from $\mathcal{A}$, then $A\cap (k+A) = \emptyset$, where $k+A\coloneqq \{k+a \pmod{n}: a\in A\}$. Moreover, when $|A| = n/2$, then $A= \overline{k + A}$, the complement of $k+A$ in $\mathbb Z/n\mathbb Z$.
        For $A\subset \Z/n\Z$ and $k \in \Z/n\Z$, we say that $A$ is \textit{$k$-self-shift complimentary} ($k$-SSC for short) if $|A| = n/2$ and $A\cap (k+A)=\emptyset$.
        
        In the following proposition, we use some elementary group theory. For $k \in \Z/n\Z$, we denote by $\langle k \rangle$ the subgroup of $\Z/n\Z$ generated by $k$, i.e., $\langle k \rangle = \{ik \pmod{n}: i \in \mathbb Z\}$. 
        Then for $j \in \Z/n\Z$, $j + \langle k \rangle$ is a coset of the subgroup $\langle k \rangle$ in $\Z/n\Z$.

        \begin{prop}[SSC structure]\label{SSC}
            Let $k \in \{1, \ldots, n-1\}$. Suppose $A\subset \Z/n\Z$ with $0\in A$, $|A|=n/2$ and $(k+A)\cap A=\emptyset$. For $d=\gcd(n,k)$, the following hold. \begin{enumerate}[(i)]
                \item $\frac{n}{d}$ is even,
                \item $(d+A)\cap A= \emptyset$, and
                \item $A$ can be partitioned into $d$ sets $O_0, \ldots, O_{d-1}$ defined as follows. For $j\in \{0,\dots,d-1\}$, we define \[
                O_j = \begin{cases}
                    j+\langle2d\rangle &\text{ if }\, j\in A,\\
                    (j+d)+\langle2d\rangle&\text{ if }\, j\notin A.
                \end{cases}
                \] 
                Each $O_j$ alternates every other entry inside of the coset $j+\langle d \rangle$. 
            \end{enumerate}  
        \end{prop}
        
        \begin{proof}
            Inside the group $\Z/n\Z$, we have $\langle k \rangle=\langle d\rangle$. Clearly we have the inclusion $\langle k \rangle\subseteq\langle d\rangle$. To get the other inclusion, use Bézout's identity to write $d=xk+yn$ for some $x,y\in \Z$ so that $d\equiv xk\pmod n$. Let us consider the cosets $\hat{O}_j := j + \langle k\rangle$ (see Figure~\ref{Ohat}). Since $\langle k \rangle = \langle d \rangle$, $\hat{O}_0, \ldots, \hat{O}_{d-1}$ partition $\Z/n\Z$. If there exists $j$ such that $|A \cap \hat O_j| > |\hat O_j|/2$, then for some $\ell \in \mathbb Z$, $j+\ell k$ and  $j + (\ell+1)k \in A$. This contradicts that $A \cap (k+A) =\emptyset$. Moreover, since $|A| = n/2$, we have that $|A \cap \hat O_j|$ is exactly $|\hat{O}_j|/2$ for all $j$. But since $A$ is $k$-SSC, this forces the elements $j,j+k,\dots,j+(\frac{n}{d}-1)k$ to alternate between being members and non-members of $A$. In particular, we have that 
    \begin{equation}\label{alternate}
                \text{for all } a\in A,  \quad a+qk\in A \iff q\equiv 0 \pmod 2.
            \end{equation}


\begin{figure}
        \centering\begin{tikzpicture}[scale=.70]
    \def\radius{3}
    \draw (0,0) circle (\radius cm);

    \foreach \i in {0,...,19} {
        \pgfmathsetmacro{\angle}{90 - 360/20 * \i} 
        \pgfmathsetmacro{\x}{\radius * cos(\angle)}
        \pgfmathsetmacro{\y}{\radius * sin(\angle)}
        \coordinate (V\i) at (\x,\y);
    }

    \draw[dotted, thick] (V0) -- (V6);
    \draw[dotted, thick] (V6) -- (V12);
    \draw[dotted, thick] (V12) -- (V18);
    \draw[dotted, thick] (V18) -- (V4);
    \draw[dotted, thick] (V4) -- (V10);
    \draw[dotted, thick] (V10) -- (V16);
    \draw[dotted, thick] (V16) -- (V2);
    \draw[dotted, thick] (V2) -- (V8);
    \draw[dotted, thick] (V8) -- (V14);
    \draw[dotted, thick] (V14) -- (V0);

    \foreach \i in {0,...,19} {
        \pgfmathsetmacro{\angle}{90 - 360/20 * \i} 
        \pgfmathsetmacro{\x}{\radius * cos(\angle)}
        \pgfmathsetmacro{\y}{\radius * sin(\angle)}
        
        \ifnum\i=0 \fill[gray!50] (\x,\y) circle (4pt); 
        \else\ifnum\i=12 \fill[gray!50] (\x,\y) circle (4pt); 
        \else\ifnum\i=4 \fill[gray!50] (\x,\y) circle (4pt); 
        \else\ifnum\i=16 \fill[gray!50] (\x,\y) circle (4pt); 
        \else\ifnum\i=8 \fill[gray!50] (\x,\y) circle (4pt); 
        \else\ifnum\i=6 \fill[white, draw=gray, thick] (\x,\y) circle (4pt); 
        \else\ifnum\i=18 \fill[white, draw=gray, thick] (\x,\y) circle (4pt); 
        \else\ifnum\i=10 \fill[white, draw=gray, thick] (\x,\y) circle (4pt); 
        \else\ifnum\i=2 \fill[white, draw=gray, thick] (\x,\y) circle (4pt); 
        \else\ifnum\i=14 \fill[white, draw=gray, thick] (\x,\y) circle (4pt); 
        \else \fill[black] (\x,\y) circle (4pt); 
        \fi\fi\fi\fi\fi\fi\fi\fi\fi\fi
        
        \node[font=\small] at (\x*1.15,\y*1.15) {\i};
    }
\end{tikzpicture}
\caption{The set $\hat O_0$ with $n = 20, k=6$. Gray points are in $A$ and white points in $\overline{A}$.}\label{Ohat}
\end{figure}
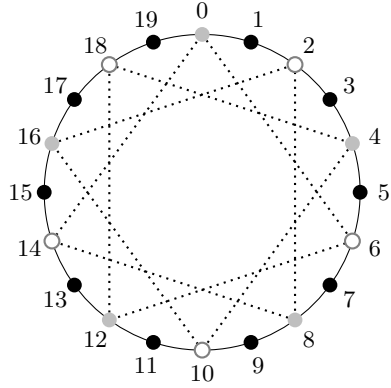
            
            If $\frac{n}{d}$ is odd then~\eqref{alternate} implies $0 + \frac{n}{d} k \notin A$. But $\frac{n}{d} k \equiv 0 \pmod{n}$, and so $0 \in A$ and $0 \notin A$, a contradiction.  This proves (i). We claim (iii) implies (ii); one checks that the definition of $O_j$ and (i) forces $(d+O_j)\cap O_j=\emptyset$. Since the $O_j$'s are a partition of $A$, (ii) follows.
            
            It now suffices to show (iii). As a direct application of \eqref{alternate}, we have the following. For every $j\in \{0,\dots,d-1\}$, we have $j+\langle 2k \rangle \subseteq A$ if and only if $j\in A$. Similarly, $(j+k)+\langle 2k \rangle \subseteq A$ if and only if $j\notin A$. This gives a partition of $A$ into the sets $O_j'$ where 
            \[
                O_j' = \begin{cases}
                        j+\langle2k\rangle &\text{ if }\,j\in A,\\
                        (j+k)+\langle2k\rangle&\text{ if }\,j\notin A.
                    \end{cases}
            \]

            
            We show that $O_j'=O_j$. First note that $\langle 2k\rangle = \langle 2d\rangle$ (again use Bézout's identity). Since $0\in A$, $k\notin A$, and thus,  $k\notin O_0' = \langle 2d\rangle$. Hence $k=qd$ for some odd $q$. Fix an arbitrary $j\in \{0,\dots,d-1\}$. If $j\in A$, then $O_j' = j+\langle 2k\rangle = j+\langle 2d\rangle = O_j$. On the other hand, if $j\notin A$ then
            \begin{align*}
                O_j' &= (j+k)+\langle 2k\rangle\\
                    &= (j+k)+\langle 2d\rangle\\
                    &= (j+qd)+\langle 2d\rangle\\
                    &= (j+d)+\langle 2d\rangle\\
                    &= O_j.
            \end{align*}
        \end{proof}
{\bf Remark}:  Proposition~\ref{SSC} also implies that a $k$-SSC set is uniquely determined by its behavior on {\em any} interval of length $d=\gcd(n,k)$, not just $\{0,\ldots d-1\}$.

\section{Proof of Theorem \ref{mainbig}  \texorpdfstring{($r \geq \lfloor \frac{n-1}{2}\rfloor$)}{}}
    
    Let $r\geq \lfloor\frac{ n-1}{2}\rfloor$ and suppose $\cH$ is an $n$-vertex, $r$-uniform hypergraph with a hamiltonian Berge cycle $C=v_0e_0v_1...e_{n-2}v_{n-1}e_{n-1}v_0$. 
    We break our proof into cases: $r > n/2, r=n/2$, and $r = \lfloor (n-1)/2 \rfloor$. The latter case can be further broken down into cases $n = 2r+1$ and $n=2r+2$ depending on the parity of $n$. 
    Recall $\E_i = \{e \in E(\cH) - E(C): v_i \in e\}$, $\U_i = \bigcup_{e \in \E_i} e$, and some vertex is contained in at least $c_r$ extra edges outside of $E(C)$.
    
    \medskip
    {\bf Case 1}: $r > n/2$. By the case, we have at least $c_r = 1$ extra edges; let $e \in \E$ be one such edge. Fix any $k+1\in \{3,...,n-1\}$. Since $|e| > n/2$, the pigeonhole principle gives $e\cap e^{+k} \neq \emptyset$. This implies $e$ contains some $k$-chord, so $\cH$ has a $(k+1)$-Berge cycle by Observation \ref{chordsgivecycles}. Since $k+1$ was arbitrary, $\cH$ has Berge cycles of every desired length and is therefore pancyclic.

\medskip
{\bf Case 2}: $r = n/2$.
	Again by the case, there exists an extra edge $e \in \E$. Suppose that $\cH$ contains no $(k+1)$-Berge cycle for some $k+1\in \{3,\dots, n-1\}$. By Observation~\ref{chordsgivecycles}, $e$ contains no $k$-chords. This gives $e \cap e^{+k}= \emptyset$. Moreover $|e| = n/2$, and so the set of indices $\{i: v_i \in e\}$ is $k$-SSC. 
    
    {\em Subcase 2.1}: 
        There exists $i$ such that $\{v_i, v_{i+1}\} \subset e$. Since $|e| < n$, we may choose $i$ such that
        $\{v_i, v_{i+1}\}\subset e$ and $v_{i-1}\notin e$ for some $i$. Without loss of generality, say $\{v_0,v_1\}\subset e$ and $v_{n-1}\notin e$. 

        We obtain another hamiltonian Berge cycle $C' = v_0ev_1\ldots v_{n-1}e_{n-1} v_0$ by swapping edges $e_0$ and $e$. With respect to $C'$, $e_0$ is an extra edge. If $e_0$ contains a $k$-chord then we get a $(k+1)$-Berge cycle by Observation~\ref{chordsgivecycles}, a contradiction. Therefore, we assume $\{i: v_i \in e_0\}$ is also $k$-SSC.
        
        Set shifting constant $s=n-k$ so that $n-s+1 = k+1$. By Lemma~\ref{Sshift}, $e_0\cap S_s(e)=\emptyset$, and since $v_{n-1}\notin e$, $|S_s(e)|=r$. Therefore,
        \[r + r = |S_s(e)| + |e_0| \leq n.\] But $r=n/2$, so we further obtain \begin{equation}\label{compliment}
        e_0 = \overline{S_s(e)},\end{equation}
        where we use $\overline{S_s(e)}$ to denote the compliment $V(\cH)- S_s(e)$. Set $d=\gcd(n,k)$. We will derive a contradiction by showing that $v_{n-i}\notin e$ if and only if $v_{n-i}\notin e_0$ for all $i\in \{0,\dots,d-1\}$. By Proposition \ref{SSC}, $k$-SSC sets are determined uniquely by their behavior on any consecutive $d$ vertices. It will then follow that $e=e_0$, a contradiction. Fix $i\in \{0,\dots,d-1\}$. Since $i<d\leq k$, we have $S_s(v_{k-i}) = v_{k-i+(n-k)}=v_{n-i}$. Then \begin{align*}
        	v_{n-i}\notin e &\iff v_{n-i+k}=v_{k-i}\in e\quad \text{since $e$ is $k$-SSC.}\\
        	&\iff S_s(v_{k-i})=v_{n-i}\in S_s(e) \quad \text{since $S_s$ is an injection on $e$.}\\
        	&\iff v_{n-i}\notin e_0 \quad\text{ by~\eqref{compliment}}.
        \end{align*}
        This gives the aforementioned contradiction.
        
    {\em Subcase 2.2}: For all $i$, $\{i,i+1\}\not\subseteq e$. Without loss of generality, let $e = \{v_0, v_2, v_4, \ldots, v_{n-2}\}$.
        Observe that $e$ contains an $\ell$-chord for all $\ell\in \{2,4,\dots , n-2\}$. Hence $\cH$ has an $(\ell+1)$-Berge cycle by Observation \ref{chordsgivecycles}. Equivalently, $\cH$ contains a Berge cycle of every odd length. We now aim to show that $\cH$ has a Berge cycle  of every even length. To do so, fix $v_j\in e_0$ such that $2\leq j \leq n-2$. Such a $j$ exists since $r=\frac{n}{2}$. 
        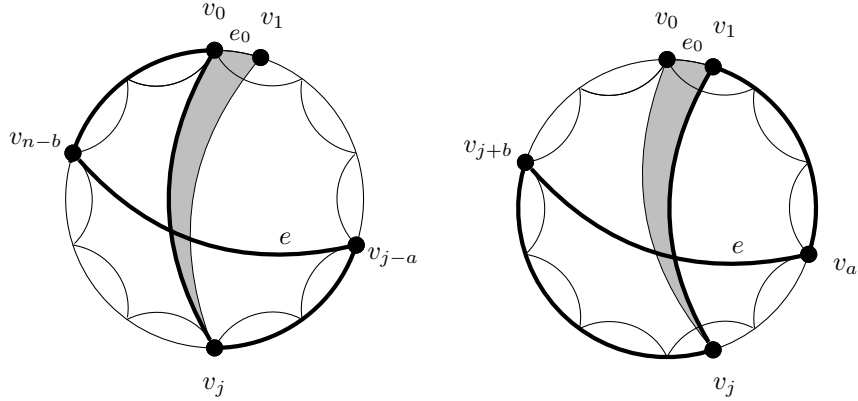
\begin{figure}[h!]
            \centering
            \begin{tikzpicture}[shorten >=1pt, minimum size=.01cm, auto, node distance=5cm, thick, line cap=rect,
                main node/.style={circle,inner sep = 2pt,  minimum width = 2pt, fill=black,draw,font=\sffamily\bfseries},
                main2 node/.style={ fill=white, inner sep = 0pt, outer sep = 0pt, font=\sffamily}
                ]
                \def\angs{0, 1, ..., 14}
                \def\specified{0/0, 1/-1, j-a/-6, j/-10, n-b/-16}
                \def \rone{1.98}
                \def \rtwo{2.5}
                \def \rthree{2.2}
                \def \rfour{1.8}
                \def \rfive{1.98}
                \node[main2 node] (e) at ({\rthree*cos(81)}, {\rthree*sin(81)}) {\small $e_0$};
                \foreach \name/\ang in \specified {
                    \node[main2 node] (v\name') at ({\rtwo*cos(18*\ang+90)},{\rtwo*sin(18*\ang+90)}) {$v_{\name}$};
                    \node[main node] (v\name) at ({\rone*cos(18*\ang+90)},{\rone*sin(18*\ang+90)}) {};
                }
                \draw[fill = lightgray, line width = 0pt] (v0) to[bend right = 30] (vj) to[bend left = 30] ({\rfive*cos(72)}, {\rfive*sin(72)}) arc (72:90:\rfive);
                \draw[ultra thick] (v0) to[bend right = 30] (vj);
                \draw[ultra thick] (vj-a) to[bend left = 30] node[pos = 0.2, above] {$e$} (vn-b);
                \draw[ultra thick,domain=270:342] plot ({\rone*cos(\x)}, {\rone*sin(\x)});
                \draw[ultra thick,domain=90:166] plot ({\rone*cos(\x)}, {\rone*sin(\x)});
                \draw[thin] (0,0) circle (\rone);
                \foreach \t in {0, 2, 4, ..., 20}{
                    \draw[thin] ({\rone*cos(18*\t+90)},{\rone*sin(18*\t+90)}) to[bend left = 45]  ({\rone*cos(18*\t+126)},{\rone*sin(18*\t+126)});
                }
                \foreach \name/\ang in \specified {
                    \node[main node] (v\name) at ({\rone*cos(18*\ang+90)},{\rone*sin(18*\ang+90)}) {};
                }
            \end{tikzpicture}
            \hspace{1em}
            \begin{tikzpicture}[shorten >=1pt, minimum size=.01cm, auto, node distance=5cm, thick, line cap=rect,
                main node/.style={circle,inner sep = 2pt,  minimum width = 2pt, fill=black,draw,font=\sffamily\bfseries},
                main2 node/.style={ fill=white, inner sep = 0pt, outer sep = 0pt, font=\sffamily}
                ]
                \def\angs{0, 1, ..., 14}
                \def\specified{0/0, 1/-1, a/-6, j/-9, j+b/-16}
                \def \rone{1.98}
                \def \rtwo{2.5}
                \def \rthree{2.2}
                \def \rfour{1.8}
                \def \rfive{1.98}
                \node[main2 node] (e) at ({\rthree*cos(81)}, {\rthree*sin(81)}) {\small $e_0$};
                \foreach \name/\ang in \specified {
                    \node[main2 node] (v\name') at ({\rtwo*cos(18*\ang+90)},{\rtwo*sin(18*\ang+90)}) {$v_{\name}$};
                    \node[main node] (v\name) at ({\rone*cos(18*\ang+90)},{\rone*sin(18*\ang+90)}) {};
                }
                \draw[ultra thick,domain=-18:72] plot ({\rone*cos(\x)}, {\rone*sin(\x)});
                \draw[ultra thick,domain=162:288] plot ({\rone*cos(\x)}, {\rone*sin(\x)});
                \draw[fill = lightgray, line width = 0pt] (v0) to[bend right = 30] (vj) to[bend left = 30] ({\rfive*cos(72)}, {\rfive*sin(72)}) arc (72:90:\rfive);
                \draw[ultra thick] (v1) to[bend right = 30] (vj);
                \draw[ultra thick] (va) to[bend left = 30] node[pos = 0.2, above] {$e$} (vj+b);
                \draw[thin] (0,0) circle (\rone);
                \foreach \t in {0, 2, 4, ..., 20}{
                    \draw[thin] ({\rone*cos(18*\t+90)},{\rone*sin(18*\t+90)}) to[bend left = 45]  ({\rone*cos(18*\t+126)},{\rone*sin(18*\t+126)});
                }
                \foreach \name/\ang in \specified {
                    \node[main node] (v\name) at ({\rone*cos(18*\ang+90)},{\rone*sin(18*\ang+90)}) {};
                }
            \end{tikzpicture}
            \caption{Even length Berge cycles.}
            \label{fig:subcase2.2}
        \end{figure}
    
        Suppose first that $j$ is even. Let $a$ and $b$ be even integers such that $0 \leq a < j$ and $0\leq  b < n-j$. The Berge cycle
        \[
            v_0e_0v_j...v_{j-a}ev_{n-b}e_{n-b+1}...v_0
        \]
        has length $a+b+2$ (see left on Figure \ref{fig:subcase2.2}). By varying $a$ and $b$, we are able to construct Berge cycles of every even length from 4 to $n-2$.
        
        Suppose instead that $j$ is odd. Let $a$ and $b$ be integers such that $a$ is even and $b$ is odd satisfying $2\leq a <j$ and $1\leq b \leq n-j$. The Berge cycle
        \[
            v_1e_1\dots v_aev_{j+b} \dots v_je_0v_1
        \]
        has length $(a-1)+b+2= a+b+1$ (see right on Figure \ref{fig:subcase2.2}). 
        Since we can find Berge cycles of every length, $\cH$ is pancyclic.
        
    \medskip
    {\bf Case 3}: $r = \lfloor (n-1)/2 \rfloor$ and $n = 2r+1$.  
        Without loss of generality, assume that $v_0$ is contained in at least $c_r=6$ extra edges as per the hypothesis of Theorem \ref{mainbig}. We begin with the following observation.
        \begin{observation}\label{kchord}
            If $|\U_i| = r+1$ and $|\E_i|\geq 3$, then for every pair $v_j, v_k \in \U_i$, there exists an $e \in \E_i$ containing $\{v_j,v_k\}$.
        \end{observation}
        Indeed, since $\cH$ is $r$-uniform, each $e\in\E_i$ misses exactly one vertex in $\U_i$. Hence, when $|\E_i|\geq 3$, every pair $\{v_j,v_k\}\in \U_i$ is contained in some edge of $\E_i$.
        
        Suppose $\cH$ does not contain a $(k+1)$-Berge cycle for some $k+1 \in \{3,\dots,n-1\}$. Set shifting constant $s = n-k$ so that $n-s+1=k+1$. Then $S_s(\U_0)\cap e_0=\emptyset$ by Lemma \ref{Sshift} part $(b)$. Hence $|S_s(\U_0)\cup e_0| = |S_s(\U_0)|+|e_0|$. 
        
        Recall that $|S_s(\U_0)|\geq |\U_0|-1$ by Lemma \ref{Sshift} part $(a)$. If $|\U_0|\geq r+3$ or $|\U_0|\geq r+2$ and $v_{n-1}\not \in \U_0$, then we have 
        \[
            n=|V(\cH)|\geq |S_s(\U_0)|+|e_0|  \geq r+2 + r =n+1,
        \]
        a contradiction. Therefore, we may assume $|\U_0|\leq r+1$ or $|\U_0|= r+2$ and $v_{n-1}\in \U_0$. Since $|\E_0| >1$, we have $|\U_0|\geq r+1$. Suppose that $|\U_0|=r+1$. By the pigeonhole principle, $\U_0\cap \U_0^{+k}\neq \emptyset$ (since $r+1>n/2$), so $\U_0$ contains some pair $\{v_i,v_{i+k}\}$. By Observation \ref{kchord},  there exists $f\in\E_0$ with $v_i,v_{i+k}\in f$,  i.e., $f$ contains a $k$-chord. Therefore, $\cH$ contains a $(k+1)$-Berge cycle, a contradiction. 
        
        Now suppose $|\U_0|= r+2$ and $v_{n-1}\in \U_0$. Note that we could reverse the order of the vertices of $C$ to make $v_1$ the $(n-1)$th vertex, so by an analogous argument, we may assume that $v_1\in \U_0$. 
        
        Let $e\in \E_0$ such that $\{v_0,v_1\}\subset e$. Set $\E_0' = \bigcup_{f\in \E_0 - e}f$ and set $\U_0'=\bigcup_{f\in \E_0'}f$. Note that $|\E_0'|=|\E_0|-1 \geq c_r-1\geq 3$. If $|\U_0'| = r+1$, then $\cH$ contains a $(k+1)$-Berge cycle by applying Observation \ref{kchord} and Observation \ref{chordsgivecycles}, a contradiction. Thus we may assume that $|\U_0'|= r+2$. 
        
        Suppose that $u\in S_s(\U_0')\cap (e_0\cup e)$. Then $u\in e$ since $S_s(\U_0)\cap e_0 =\emptyset$. We may swap $e$ and $e_0$ to obtain another hamiltonian Berge cycle $C'=v_0ev_1 \dots v_{n-1}e_{n-1}v_0$ since $e$ contains $v_0$ and $v_1$. Therefore, $\cH$ contains a $(k+1)$-Berge cycle by Lemma \ref{Sshift} part $(b)$ since $S_s(\U_0')\cap e\neq \emptyset$, a contradiction. Thus $S_s(\U_0')\cap(e_0\cup e)=\emptyset$, so it follows that 
        \begin{align*}
            n\geq |V(\cH)|  &\geq |S_s(\U'_0)\cup (e_0\cup e)| = |S_s(\mathcal{U}_0')| + |e_0\cup e| \geq (r+1) + (r+1)=n+1,
        \end{align*}
        a contradiction. 
    
    \medskip
    {\bf Case 4}: $r = \lfloor (n-1)/2 \rfloor$ and $n = 2r+2$. As in Case 3, we assume $v_0$ is contained in $|\E_0| \geq c_r=6$ extra edges. Suppose for some $k+1\in \{3,\dots n-1\}$ that $\cH$ is missing a $(k+1)$-Berge cycle. 
    
    The following is a brief outline of Case 4. 
        We will show that $\U_0$ cannot be too large since $S_s(\U_0)$ would forbid too many vertices from $e_0$. Moreover, if $\U_0$ is large, then $\U_0$ will contain several pairs of the form $\{v_i,v_{i+k}\}$. Using a combination of these two ideas, we are able to remove a small number of edges from $\E_0$ in order to decrease the size of $\U_0$. Eventually, we remove enough edges from $\E_0$ such that the remaining edges cover exactly $r+1$ vertices. Formally, we construct a subset $\E_0'\subseteq \E_0$ with $\U_0'=\cup_{e\in \E_0'}e$ where $|\E_0'|\geq 3$ and $|\U_0'|=r+1$. 
        Since $r+1=n/2$, we then apply our SSC Lemma to find Berge cycles in $\cH$ using the rigid structure of the extra edges. 
        
        \begin{claim}\label{tor+2}
        	There exists a subset $\E_0'\subseteq \E_0$ such that
            \[
                |\E_0' | \geq \left\lceil \frac{|\E_0|-1}{2} \right\rceil\text{ and }|\bigcup_{e\in \E_0'}e|=r+1.
            \]
        \end{claim}
        \begin{proof}
            First we find $\E_0^*\subseteq \E_0$ such that $|\bigcup_{e\in \E_0^*}e| \leq r+2$. Set shift constant $s=n-k$ so that $n-s+1=k+1$. Since we assumed $\cH$ does not contain a $(k+1)$-Berge cycle, $S_s(\U_0)\cap e_0 = \emptyset$ by Lemma \ref{Sshift} part $(b)$. Suppose that $|\U_0|\geq r+4$ or $|\U_0|\geq r+3$ and $v_{n-1}\not \in \U_0$. Then by Lemma \ref{Sshift} part $(a)$, we have 
            \begin{align*}
                n=|V(\cH)|  &\geq |S_s(\U_0)\cup e_0| =|S_s(\U_0)|+ |e_0| \geq r+3 +r = n+1,
            \end{align*}
            a contradiction. Hence, $ |\U_0|\leq r+2$ or $ |\U_0|= r+3$ and $v_{n-1}\in \U_0$. 
            If $|\U_0|\leq r+2$, then set $\E_0^*=\E_0$ with $\U_0^*=\bigcup_{e\in\E_0^*}e$. If $|\U_0|= r+3$, then by applying a symmetric argument, we may assume $v_1\in \U_0$. Fix an extra edge $e\in \E_0$ with $v_1\in e$ and set $\E_0^* = \E_0-\{e\}$ and $\U_0^* = \bigcup_{f\in \E_0^*} f$. Suppose that $|\U_0^*|=r+3$. As in the previous case, we have that $S_s(\U_0^*)\cap (e_0\cup e)=\emptyset$, and therefore
            \begin{align*}
                n = |V(\cH)|&\geq |S_s(\U_0^*)\cup (e\cup e_0)|=|S_s(\U_0^*)|+|e\cup e_0|\geq (r+3-1)+(r+1)=n+1,
            \end{align*}
            a contradiction; hence, $|\U_0^*|\leq r+2$. Note that $|\E_0^*|\geq |\E_0|-1\geq 5$.

            We now aim to find $\E_0'\subseteq \E_0^*$ such that $|\bigcup_{e\in \E_0'}e|=r+1$. If $|\U^*_0|<r+2$, then we can take $\E_0'=\E_0^*$, so we assume $|\U^*_0|=r+2$. Observe that 
        \begin{align*}
            |\U_0^* \cap (\U_0^*)^{+k}| &= |\U^*_0|+ |(\U_0^*)^{+k}| - |\U_0^* \cup (\U_0^*)^{+k}|\\
            	&\geq 2|\U_0^*| -n \\
            	&= 2(r+2) - (2r+2)=2.
        \end{align*}
        Therefore, $\U^*_0$ contains two pairs of vertices of the form $\{v_i, v_{i+k}\}$ and $\{v_j, v_{j+k}\}$. Note that these four vertices need not be distinct. Let $T = \{v_i, v_{i+k},v_j, v_{j+k}\}$. Then $|T|\in \{2,3,4\}$. 
        	
        \textit{Case 1: $|T|=4$.} In this case, $|\U^*_0- T|=r-2$, so every edge in $\E^*_0$ contains at least 2 vertices from $T$. If $|f\cap T|\geq 3$ for some $f\in\E_0^*$, then $f$ must contain either $\{v_i,v_{i+k}\}$ or $\{v_j,v_{j+k}\}$ as a $k$-chord, which creates a $(k+1)$-Berge cycle by Observation \ref{chordsgivecycles}, a contradiction. Hence, each edge contains exactly two vertices of $T$ and all $r-2$ vertices of $\U_0^*- T$. Observe that $\E_0^*$ could have at most 4 edges which avoid both $\{v_i,v_{i+k}\}$ and $\{v_j, v_{j+k}\}$. Since $|\E_0^*|>4$, some edge in $f \in \E_0^*$ contains a $k$-chord. Therefore, $\cH$ contains a $(k+1)$-Berge cycle, a contradiction.
        
        \textit{Case 2: $|T|=3$.} Without loss of generality, let $T= \{v_i, v_{i+k}, v_{i+2k}\}$ with $i\not\equiv i+2k\pmod n$. Then $|\U_0^*- T|=r-1$, so every edge in $\E_0^*$ contains at least 1 vertex from $T$. Assume no edge from $\E_0^*$ contains either $\{v_{i}, v_{i+k}\}$ or $\{v_{i+k},v_{i+2k}\}$, otherwise we obtain a $k$-chord. If any edge of $\E_0^*$ intersects $T$ in 2 or more vertices, then it does not contain $v_{i+k}$. 
        So exactly one edge in $\E_0^*$, say $f$, contains $v_{i+k}$. Set $\E_0' = \E^*_0 - \{f\}$. This set satisfies $|\E_0'| = |\E^*_0| -1\geq |\E_0| -2 \geq \lceil \frac{|\E_0|-1}{2} \rceil$ (since $|\E_0| \geq 3$) and $|\bigcup_{e\in \E_0'}e| = |\U_0^*| - |\{v_{i+k}\}| = r+1$, as desired. 
    

        \textit{Case 3: $|T|=2$.} Let $T=\{v_i,v_{i+k}\}$ where $2k=n$. Here we have $|\U_0^*- T|=r$ and we may assume no edge in $\E_0^*$ contains both $v_i$ and $v_{i+k}$ since this would create a $(k+1)$-Berge cycle. There is at most one edge in $\E_0^*$ which contains neither $v_{i}$ nor $v_{i+k}$, say edge $g$. The remaining edges in $\E_0^*-\{g\}$ contain exactly one of $v_i$ or $v_{i+k}$. 
        Without loss of generality, at most half of edges in $\E_0^*-\{g\}$ contain $v_i$ and we remove them to construct $\E_0'$. If the edge $g$ exists, then $|\E_0'|\geq \lceil \frac{|\E_0^*| - 1}{2}\rceil + 1 \geq \lceil \frac{|\E_0|-2}{2} \rceil +1\geq \lceil \frac{|\E_0|}{2} \rceil$; otherwise, $|\E_0'| \geq \lceil \frac{|\E^*_0|}{2} \rceil \geq \lceil \frac{|\E_0|-1}{2} \rceil$, and $|\bigcup_{e\in\E_0'}e|=r+1$.
    \end{proof}
        
Next we prove the following claim which is an extension of Observation \ref{kchord}.           
        
    \begin{claim}\label{matchedge}
        Fix $j$ distinct pairs $\{\{a_i,b_i\}\}_{i \in [j]}$ of vertices in $\U_0'$ such that each vertex is in at most $|\E_0'|-1$ pairs. If $j \leq |\E_0'|$, then there exists $j$ distinct edges $\{e(a_i,b_i)\}_{i \in [j]} \subseteq \E_0'$ such that $\{a_i, b_i\} \subset e(a_i,b_i)$ for all $i\in[j]$. 
    \end{claim} 
    \begin{proof}
        Since $|\U_0'| = r+1$, each edge in $\E_0'$ is determined uniquely by the element in $\U_0'$ that it omits. For $u\in \U_0'$ define $e_u= \U_0' - \{u\}$. Construct a bipartite graph $G = X \sqcup Y$, where $X = \{\{a_i,b_i\}\}_{i \in [j]}$, $Y = \E_0'$ and vertex $\{a_i,b_i\} \in X$ and vertex $e_u \in Y$ are adjacent if and only if $\{a_i,b_i\} \subset e_u$, i.e., $u \notin \{a_i,b_i\}$. Therefore, every vertex in $X$ has degree at least $|\E_0'|-2$. If $G$ has a matching saturating $X$, then we are done. Otherwise there exists $X' \subseteq X$ with $|X'| > |N_G(X')|$ by Hall's Theorem. In particular, we must have $|X'| \geq |\E'_0|-1 \geq 2$, and for any non-neighbor $e_u$ of $X'$ and every $\{a_i,b_i\} \in X'$, $\{a_i,b_i\} \not\subset e_u$. That is, each pair of $X'$ intersects the vertex $u$. 
        
        If $Y - N_G(X')$ contains only a single element, $e_u$, then $|X'|>|N_G(X')|$ implies $X\geq |\E_0'|$. Then there are $|\E_0'|$ pairs containing the vertex $u$, a contradiction. If $e_u, e_{u'} \in Y- N_G(X')$, then every pair of $X'$ contains both $u$ and $u'$. Since $|X'|\geq 2$, the pairs of $X$ are not distinct, a contradiction. 
    \end{proof}
    Throughout our proof, we will use $\{e(a_i,b_i)\}_i$ to refer to such a set of distinct edges in $\E_0'$. We are now ready to prove Case 4, which will finish the proof Theorem \ref{mainbig}. 

    Let $\E_0'$ and $\U_0'$ be the sets produced by Claim \ref{tor+2}. Recall $\E'_0$ are extra edges of $\cH$ containing $v_0$. Note that $|\E_0'|\geq 3$ since $|\E_0|\geq 6$ and $|\U_0'|=r+1= n/2$. By assumption, $\cH$ does not contain a $(k+1)$-Berge cycle. Hence $\cH$ does not contain any $k$-chords. Then $\{i: v_i \in\U_0'\}$ is a $k$-SSC set. We leverage this structure to build Berge cycles similar to those built in Case 2 when $r=n/2$. 
    
    Let $d = \gcd(n,k)$. Since $v_0 \in \U_0'$, we have $v_{k} \notin \U_0'$. Let $m, p > 0$ be the least positive integers such that $v_{k+p} \in \U_0'$ and $v_{k-m} \in \U_0'$. By Proposition~\ref{SSC} (ii), $v_{k+d}, v_{k-d} \in \U_0'$, so $1\leq p, m \leq d$. Moreover, since $v_{k-(m-1)} \notin \U_0'$, we have $v_{-(m-1)} \in \U_0'$.
    
    Suppose first that $|\{k-m, k+p, -(m-1)\}| = 3$. The Berge cycle
    \[v_{-(m-1)} e_{-(m-1)} v_{-(m-2)} \ldots v_{k-m} e(v_{k-m}, v_{k+p}) v_{k+p} e(v_{k+p}, v_{-(m-1)}) v_{-(m-1)}\] has length $k + 1$ since the segment from $v_{-(m-1)}$ to $v_{k-m}$ has length $k-1$, and we use 2 additional edges to connect $v_{k+p}$. The distinct edges $e(v_{k-m}, v_{k+p}), e(v_{k+p}, v_{-(m-1)})$ exist by Claim~\ref{matchedge}.
    
    Now suppose $|\{k-m, k+p, -(m-1)\}| = 2$. Observe that $d=\gcd(n,k) \leq \min\{k, n-k\} \leq n/2$. Since $m, p \leq d$, if $k-m = k+p$, then $d = k=n/2$ and $k-m = k+p = 0$. This implies $|\U_0'| = 1$, a contradiction. Moreover we cannot have $k-m = -(m-1)$, otherwise $k=1$. So we may assume $k+p = -(m-1)$.
    
    Partition $\{0, \ldots, n-1\}$ into $n/d$ intervals of length $d$: $\{0, \ldots, d-1\}, \{d, \ldots, 2d-1\}, \ldots, \{n-d, \ldots, n-1\}$. 
    The element $k+p = -(m-1) =n-(m-1)$ belongs to the last interval of length $d$. This interval must be equal to $\{k, k+1, \ldots, k+d-1\}$ since it contains $k+p$. Therefore, $n = k+d$. 
    
    By the choice of $p$ and $m$, the set $\{v_{k-(m-1)}, v_{k-(m-2)}, \ldots, v_k, \ldots, v_{k+(p-1)}\}$ contains $m-1 + p$ consecutive elements that do not belong to $\U_0'$. But $k + p = n-(m-1)$, which implies $m-1 + p = n-k = d$. 
    Since $\{i: v_i\in \U_0\}$ is $k$-SSC, there exists some interval of $d$ consecutive vertices contained in $\U_0'$ which we will denote by $\{v_i, v_{i+1}, \ldots, v_{i+(d-1)}\} \subseteq \U_0'$.
    
    First suppose $k \neq n/2$. There exists two disjoint intervals $\{v_i, v_{i+1}, \ldots, v_{i+(d-1)}\}\subseteq \U_0'$ and $\{v_{i+2d}, v_{i+2d+1}, \ldots, v_{i+2d + (d-1)}\} \subseteq \U_0'$. The Berge cycle
    \[
        v_i e(v_i, v_{i+d-1}) v_{i+d-1} e_{i+d-1} \ldots v_{i + 2d} e(v_{i+2d}, v_{i+2d+2}) v_{i+2d+2} \ldots v_i
    \]
    contains $n - (d-2) - 1 = n-d+1 = k+1$ vertices, a contradiction. 
    
    Finally suppose $k=d= n/2$. Without loss of generality, assume $\{v_0, \ldots, v_{k-1}\} = \U_0'$. The edge set $\{e_0, \ldots, e_{k-2}\} \cup \E_0'$ contains at least $k+2 = n/2+2= \binom{|\U_0'|}{r}+2$ edges. Thus at least two edges in $\{e_0, \ldots, e_{k-2}\}$ are not contained in $\U_0'$. Let $e_j$ with $0< j \leq k-2$ be such an edge, with say $v_{k+s} \in e_j$, where $0\leq s \leq k-1$. Suppose $s \leq k/2$ (the case where $s > k/2$ is symmetric because we can rotate/reorient $C$ as $v_{k-1}e_{k-2}v_{k-2}\ldots v_0 e_{n-1} \ldots v_{k-1}$). Let $t$ be an integer such that \[0\leq t < j< t+ s + 2 \leq k-1.\] Since $0\leq s \leq k/2$, such a $t$ exists. The Berge cycle
    \[C'=v_0 \ldots v_t e(v_t, v_{t+s+2}) v_{t+s+2} e_{t+s+2} \ldots v_{k+s} e_j v_j e(v_j,v_0) v_0\] 
    contains the $s + 1$ vertices from $k$ to $k+s$ and the $k - (s+1) + 1$ vertices in $(\{v_0, \ldots, v_{k-1}\} -\{v_{t+1}, \ldots, v_{t+s+1}\}) \cup \{v_j\}$. That is, $|C'| = k+1$.
    This completes the proof of Theorem \ref{mainbig}. \qed
        

    
\section{Proof of Theorem~\ref{mainsmall}  \texorpdfstring{($r = \lfloor \frac{n-1}{2}\rfloor -1$)}{}}

Here we have $n\in \{2r+3,2r+4\}$ and every vertex is in at least $5(r-1)+2$ extra edges. Recall for each $v_i \in V(\cH)$, $\E_{i} = \{e \in \E: v_i \in e\}$ and $\U_i = \bigcup_{e\in \E_{i}} e$. Since at most $\binom{r+1}{r} < 5(r-1)+2$ edges can fit inside $r+1$ vertices, we must have $|\U_i|\geq r+2$.

We begin with a brief outline of the proof. We will construct a graph $G$ in which cycles of length $\ell$ yield Berge cycles of the same length in $\cH$. The graph $G$ will contain pairs with large co-degree in $\cH$. We then show that $G$ is pancyclic by applying known results for pancyclic and weakly pancyclic graphs.

Suppose that $\cH$ does not contain a $(k+1)$-Berge cycle for some $k+1 \in \{3, \ldots, n-1\}$. We set the shifting constant $s = n-k$ and apply Lemma \ref{Sshift} to each $v_i$ (we may rotate $C$ to view $v_i$ as $v_0$) to obtain
\[
|e_i|+|\U_i|-1\leq |e_i|+|S_s(\U_i)| \leq n.   
\]

Since $|e_i|=r$, we have $|\U_i|\leq r+5$. Thus,  
\begin{equation}\label{smallcodegree}
r+2 \leq |\U_i|\leq r+5.    
\end{equation}

\begin{claim}\label{largecodegree}
For any $\U_i$ satisfying \eqref{smallcodegree}, there are at least $r+1$ vertices $v_j \in \U_i-\{v_i\}$ such that $d_\E(v_j,v_i)\geq r$.
\end{claim}

\begin{proof}
Let $|\U_i|=r+c$ where $2\leq c \leq 5$. It is sufficient to show that there are at most $c-2$ vertices $v_j \in \U_i-\{v_i\}$ such that $d_\E(v_j,v_i) \leq r-1$. Suppose there are $c-1$ such vertices $v_{j_1},\ldots,v_{j_{c-1}} \in \U_i-\{v_i\}$. Then $\U_i - \{v_i\} - \{v_{j_1}, ..., v_{j_{c-1}}\}$ contains $r$ vertices, and hence $\E_i$ contains at most $\binom{r}{r-1} = r$ edges that do not intersect $\{v_{j_1}, ...., v_{j_{c-1}}\}$. Thus,
\[5(r-1)+2 \leq d_\E(v_i) \leq \sum_{k=1}^{c-1} d_\E(v_{j_k},v_i)+r\leq (c-1)(r-1)+r =c(r-1) + 1,\]
a contradiction.\end{proof}


Suppose $G$ and $\cH$ are a graph and hypergraph, respectively, over vertex set $V(\cH)$. We say $G$ and $\cH$ are {\em cycle compatible} if for any cycle $D=u_1u_2\ldots u_{\ell}u_1$ in $G$ there exists a matching $\phi: E(D) \to E(\cH)$ such that $u_1 \phi(u_1u_2) u_2 \phi(u_1u_2)\ldots u_{\ell} \phi(u_{\ell}u_1) u_1$ is a Berge cycle in $\cH$ of the same length. In particular, if $G$ and $\cH$ are cycle compatible and $G$ is pancyclic, then $\cH$ is also pancyclic.

\begin{claim}\label{cyclecompatible}
Let $G$ be a graph with $V(G) = V(\cH)$ such that $d_{\cH}(x,y) \geq 1$ for every $xy \in E(G)$. Suppose that $E(G)$ has a partition $E_1 \cup E_2$ satisfying
\begin{enumerate}
    \item there exists an injection $\phi^*: E_1 \to E(\cH)$ such that for every $xy \in E_1, \{x,y\} \subseteq \phi^*(xy)$ (we call $E_1$ the fixed edges), and
    \item for every $xy \in E_2, d_{E(\cH) - \phi^*(E_1)}(x,y) \geq r$.    
\end{enumerate}
Then G and $\cH$ are cycle compatible.
\end{claim}

\begin{proof}
Let $D=u_1u_2\ldots u_{\ell}u_1$ be a cycle of length $\ell$ in $G$. We will establish a matching $\phi: E(D) \to E(\cH)$. Define $E_1(D)=E_1 \cap E(D)$ and $E_2(D)=E_2 \cap E(D)$. First set $\phi|_{E_1(D)}=\phi^*$. Now we establish a matching $\phi|_{E_2(D)}: E_2(D) \to E(\cH) - \phi^*(E_1(D))$. Consider the bipartite graph $F=A \sqcup B$ where $A=E_2(D)$ and $B=E(\cH) -\phi^*(E_1(D))$.
For $a=u_iu_{i+1}\in A, b=e\in B$, let $ab \in E(F)$ if and only if $\{u_i,u_{i+1}\} \subseteq e$.
We observe that $d_F(a)\geq r$ for all $a \in A$ because of item 2. Every $e \in E(\cH)$ contains at most $r$ pairs from $\{u_1u_2,u_2u_3,\ldots,u_{\ell}u_1\}$, so $d_F(b)\leq r$ for all $b \in B$. For any $S \subseteq A$, the number of edges between $S$ and $N_F(S)$ satisfies
\[r|S|\leq |E(S,N_F(S))|\leq r|N_F(S)|.\]
This gives $|S|\leq |N(S)|$. Thus by Hall's Theorem there is a matching $\phi_2$ of size $|A|=|E_2(D)|$ between $A$ and $B$. Set $\phi|_{E_2(D)}=\phi_2$. 
Then $u_1 \phi(u_1u_2) u_2 \phi(u_1u_2)\ldots u_{\ell} \phi(u_{\ell}u_1) u_1$ is a Berge cycle of length $\ell$ and hence G and $\cH$ are cycle compatible. 
\end{proof}

Let $C= v_0e_0v_1...v_{n-1}e_{n-1}v_0$ be the Berge hamiltonian cycle of $\cH$. We construct a graph $G$ over vertex set $V(\cH)$ where $xy \in E(G)$ if $d_{\E}(x,y) \geq r$ or if $xy = v_iv_{i+1}$ for some $i \in\{0, \ldots, n-1\}$.

Let $E_1=\{v_iv_{i+1} \in E(G): i \in \{0,\ldots,n-1\}\}$ and $E_2=\{v_iv_j \in E(G): |j-i|>1\}$. Define a bijection $\phi^*:E_1 \to E(C)$ such that for every $v_iv_{i+1} \in E_1$, $\phi^*(v_iv_{i+1})=e_i$. Since $E(\cH) -\phi^*(E_1)=E(\cH) - E(C)=\E$, for every $v_iv_j \in E_2$, $d_{\E}(v_i,v_j)=d_{E(\cH) -\phi^*(E_1)}(v_i,v_j)\geq r$. Then Claim~\ref{cyclecompatible} implies $G$ and $\cH$ are cycle compatible.


By Claim~\ref{largecodegree}, $\delta(G) \geq r+1 \geq (n+2)/3$. 
If $G$ contains a triangle, then it is weakly pancyclic by Theorem~\ref{Brandt}. Moreover, $G$ is hamiltonian and therefore pancyclic. This implies $\cH$ is pancyclic, and we are done.
So suppose $G$ is triangle-free (and non-pancyclic).

If $d_{\E}(v_i,v_{i+1}) \leq r-1$ for all $i$, then by Claim \ref{largecodegree}, $d_G(v_i)\geq r+1+|\{v_{i-1}, v_{i+1}\}|=r+3 > \frac{n}{2}$. Theorem \ref{Bondy} implies that $G$ (and therefore $\cH$) is pancyclic, a contradiction.

Now assume $d_\E(v_t,v_{t+1}) \geq r$ for some $t$. Define the set \[X_t=\{v_{h-1},v_{h+1}: v_h \in e_t-\{v_t,v_{t+1}\}\}.\] We will show that we can obtain a new graph $G'$ by adding one edge to $G$ such that $G'$ has a triangle and is cycle compatible with $\cH$.

\begin{claim}\label{Et1}
$X_t \cap N_G(v_t) = \emptyset$.
\end{claim}

\begin{proof}
Suppose $X_t \cap N_G(v_t) \ne \emptyset$ and without loss of generality there exists an index $h$ such that $v_h \in e_t-\{v_t,v_{t+1}\}$ and $v_{h+1} \in N_G(v_t)$. 
Since we assumed $G$ was triangle-free and $v_tv_{h+1}, v_hv_{h+1} \in E(G)$, $v_tv_h \notin E(G)$. Set $G'=G + v_tv_h$.

Since $G'$ also contains a hamiltonian cycle and $\delta(G')\geq r+1 \geq \frac{n+2}{3}$, Theorem \ref{Brandt} implies $G'$ is pancyclic.
Moreover $G'$ contains the triangle $v_tv_hv_{h+1}$ so $G'$ is pancyclic.
Now we show that $G'$ and $\cH$ are cycle compatible. Let $E'_1=E_1\cup \{v_tv_h\}-\{v_tv_{t+1}\}$ and $E'_2=E_2 \cup \{v_tv_{t+1}\}$. Define a bijection $\phi^*:E'_1 \to E(C)$ such that $\phi^*(v_iv_{i+1})=e_i$ for every $v_iv_{i+1} \in E_1-\{v_tv_{t+1}\}$, and $\phi^*(v_tv_h)=e_t$.  Since $E(\cH)-\phi^*(E'_1)=E(\cH)-E(C)=\E$, and $d_{\E}(v_{t}v_{t+1}) \geq r$, for every $v_iv_j \in E'_2$, $d_{\E}(v_i,v_j)=d_{E(\cH)-\phi^*(E'_1)}(v_i,v_j)\geq r$. Claim~\ref{cyclecompatible} implies $G'$ and $\cH$ are cycle compatible. Then $\cH$ is pancyclic, a contradiction.
\end{proof}

\begin{claim}
$|X_t| \geq 2(r-2)$.    
\end{claim}

\begin{proof}
Suppose for contradiction there exists some index $h$ such that $v_h,v_{h+2} \in e_t-\{v_t,v_{t+1}\}$. Since $G$ is triangle-free and $v_hv_{h+1}, v_{h+1}v_{h+2} \in E(G)$, we must have $v_hv_{h+2} \notin E(G)$. 
Let $G'=G+v_hv_{h+2}$. Then $G'$ is weakly pancyclic and contains the triangle $v_hv_{h+1}v_{h+2}$, hence  is pancyclic.

Now we show that $G'$ and $\cH$ are cycle compatible. Let $E'_1=E_1\cup \{v_hv_{h+2}\}-\{v_tv_{t+1}\}$ and $E'_2=E_2 \cup \{v_tv_{t+1}\}$. Define a bijection $\phi^*:E'_1 \to E(C)$ such that $\phi^*(v_iv_{i+1})=e_i$ for every $v_iv_{i+1} \in E_1-\{v_tv_{t+1}\}$, and $\phi^*(v_hv_{h+2})=e_t$.  
Note $E(\cH)-\phi^*(E'_1)=E(\cH)-E(C)=\E$ and $d_{\E}(v_t,v_{t+1}) \geq r$. Hence $d_{\E}(v_i,v_j)=d_{E(\cH)-\phi^*(E'_1)}(v_i,v_j)\geq r$ for every $v_iv_j \in E'_2$.
Then Claim~\ref{cyclecompatible} implies $G'$ and $\cH$ are cycle compatible and thus $\cH$ is pancyclic. Hence, it is impossible that both $v_h$ and $v_{h+2}$ belong to $e_t-\{v_t,v_{t+1}\}$ for any $h \in\{0, \ldots, n-1\}$. In particular, for all $v_h, v_{h'} \in e_t -\{v_t,v_{t+1}\}$, we have $\{v_{h-1}, v_{h+1}\} \cap \{v_{h'-1}, v_{h'+1}\} = \emptyset$. Then $|X_t| \geq 2(r-2)$.
\end{proof}

Claim \ref{Et1} implies
\[|X_t|+|N_G(v_t)|\leq n \leq 2r+4.\]
Since $|X_t| \geq 2(r-2)$ and $|N_G(v_t)|\geq r+1$, 
\[2(r-2)+r+1\leq 2r+4.\]
This is a contradiction for $r\geq 8$, i.e., $n \geq 2r+3 \geq 19$.
\qed
\bigskip

\section{Concluding remarks}
For small $r$, we can construct hamiltonian, non-pancyclic $r$-uniform hypergraphs as follows.

    {\bf Construction 4}. For $k \geq 3$, start with the $k$-cycle (in the graph sense) $v_1, v_2, \ldots, v_k, v_1$ and replace each edge $v_iv_{i+1}$ with an $r$-uniform clique of size $r+1$ containing $v_i, v_{i+1}$ and $r-1$ distinct new vertices. Let $\cH_{k,r}$ be the resulting hypergraph on $n=kr$ vertices. 

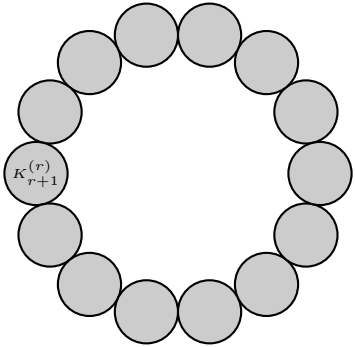
\begin{center}
\begin{tikzpicture}[scale=0.7]

\def\n{14}
\def\r{0.6} 
\def\R{\r/(sin(180/\n))} %
\foreach \i in {1,...,\n} {
    \pgfmathsetmacro{\angle}{360/\n * \i}
    \pgfmathsetmacro{\x}{\R * cos(\angle)}
    \pgfmathsetmacro{\y}{\R * sin(\angle)}
    \ifnum\i=7 
        \draw[fill=black!20, draw=black, thick] (\x,\y) circle (\r); 
        \node at (\x,\y) {\tiny{$K_{r+1}^{(r)}$}};
    \else
        \draw[fill=black!20, draw=black, thick] (\x,\y) circle (\r); 
    \fi
}

\end{tikzpicture}
\end{center}
     Observe that each $r+1$-clique contains a $v_i,v_{i+1}$-Berge path of any length from $1$ to $r$. Combining these paths, we obtain Berge cycles of every length from $k=n/r$ to $n$. Moreover, inside the cliques we have Berge cycles of lengths $2$ to $r+1$.
     
     On the other hand, if $k=n/r \geq  r+3$, then $\cH_{k,r}$ contains no Berge cycles with lengths between $r+2$ and $k-1 = n/r - 1$ (in particular, the $2$-shadow of $\cH_{k,r}$ contains no cycles of these lengths).
     This construction fails to be pancyclic when $n \geq r^2 + 3r$. We are not aware of constructions that are hamiltonian but not pancyclic for larger $r$. 

     In Theorem~\ref{mainbig}, we proved that for large $r$, the minimum-degree condition in Theorem~\ref{main} is not necessary to guarantee that a hamiltonian hypergraph is pancyclic. Instead, if some vertex is contained in sufficiently many (6) edges outside of the Berge hamiltonian cycle, then we can make Berge cycles of any length using the chords of these large edges.
     
     This condition of having $6$ extra edges can likely be further loosened. We propose the following questions.

     {\bf Question 1}. Let $n$ and $r$ be integers with $r = \lfloor (n-1)/2 \rfloor$ and $n$ sufficiently large. Suppose $\cH$ is an $n$-vertex $r$-uniform hypergraph with a hamiltonian cycle $C$. If $\cH$ contains at least one edge outside of $C$, is it true that $\cH$ is pancyclic?

     {\bf Question 2}. For which uniformities $r = r(n)$ is it true that when $n$ is sufficiently large, every $r$-uniform, $n$-vertex hamiltonian hypergraph is also pancyclic?

Entringer and Schmeichel~\cite{ES} proved that every $2n$-vertex hamiltonian bipartite graph with more than $n^2/2$ edges is bipancyclic.~\footnote{We are unfortunately unable to find their paper online.} We translate the hypergraph problem into bipartite graphs by considering the bipartite incidence graph of an $n$-vertex, $r$-uniform hamiltonian hypergraph with $n$ edges. If $r > n/2$, then the corresponding bipartite graph is hamiltonian with exactly $rn > n^2/2$ edges. This shows that all $r > n/2$ satisfy Question 2. (Note that this also implies  Theorem~\ref{main} for $r > n/2$, although our proof for this case was also trivial.)  On the other hand, Construction $4$ demonstrates that it fails when $r \mid n$ and $n \geq r^2+3r$, i.e., $r = \sqrt{n}(1-o(1))$. It would be interesting to study what happens between these bounds.
     
{\bf Acknowledgements}. The authors would like to thank Jonah Klein and Dinesh Limbu for their helpful comments.

\end{document}